%% file: extension.tex
\documentclass[DIV=classic,a4paper,10pt]{myart}

\KOMAoptions{DIV=last}
\DeclareOldFontCommand{\tt}{\normalfont\ttfamily}{\mathtt}
\input{head2013}

\begin{document}
 
\title{Vector duality via conditional extension of dual pairs}

\author[a,1,t2]{Samuel Drapeau}
\author[b,2,t1,t3]{Asgar Jamneshan}
\author[b,3,t1]{Michael Kupper}

\address[a]{Shanghai Jiao Tong University, Shanghai Advanced Institute of Finance}
\address[b]{University of Konstanz}
\eMail[1]{sdrapeau@saif.sjtu.edu.cn}
\eMail[2]{asgar.jamneshan@uni-konstanz.de}
\eMail[3]{kupper@uni-konstanz.de}

\myThanks[t1]{Financial support from the DFG Project KU-2740/2-1 is gratefully acknowledged.}
\myThanks[t2]{Financial support from the National Science Foundation of China ''Research Fund for International Young Scientists'' Grant number 11550110184 is gratefully acknowledged.}

\myThanks[t3]{The author would like to thank Jonathan Borwein for his kind encouragement.}

\abstract{
A Fenchel-Moreau type duality for proper convex and lower semi-continuous functions $f\colon X\to \overline{L^0}$ is established where $(X,Y,\langle \cdot,\cdot \rangle)$ is a dual pair of Banach spaces and $\overline{L^0}$ is the set of all extended real-valued measurable functions. 
We provide a concept of lower semi-continuity which is shown to be equivalent to the existence of a dual representation in terms of elements in the Bochner space $L^0(Y)$. 
To derive the duality result, several conditional completions and extensions are constructed.  
\vspace{7mm}
}
\keyWords{Vector duality, extensions, conditional sets, Bochner spaces, vector optimization}

\keyAMSClassification{46A20, 03C90, 54D35, 54C20, 46B22}

\date{\today}
\ArXiV{}
\maketitle

\section{Introduction}

Duality theory is an important tool in vector optimization and its wide range of applications. 
This article contributes to vector duality by providing a notion of lower semi-continuity and proving its equivalence to a Fenchel-Moreau type dual representation. 
For the sake of illustration, let $(\Omega,\mathcal{F},\mu)$ be a $\sigma$-finite measure space, $(X,Y,\langle \cdot, \cdot\rangle)$ a dual pair of Banach spaces and $\overline{L^0}$ the collection of all measurable functions $x\colon\Omega\to \mathbb{R}\cup\{\pm\infty\}$ where two of them are identified if they agree almost everywhere. 
Consider on $\overline{L^0}$ the order of almost everywhere dominance. 
Let $f\colon X\to \overline{L^0}$ be a proper convex function. 
We introduce a concept of lower semi-continuity which is shown to be equivalent to the Fenchel-Moreau type dual representation 
\begin{equation}\label{darstellung}
 f(x)=\esssup_{y\in L^0(Y)}\{\langle x,y\rangle - f^\ast(y)\} 
\end{equation}
for all $x\in X$, where $L^0(Y)$ is the Bochner space of all strongly measurable functions $y\colon \Omega\to Y$ modulo almost everywhere equality, $f^\ast(\cdot)=\esssup_{x\in X}\{\langle x,\cdot\rangle - f(x)\}$ is the convex conjugate and $\langle x,y\rangle(\omega)=\langle x,y(\omega)\rangle$ almost everywhere. 

The basic idea is to conditionally extend the whole structure of the problem to conditional set theory where a conditional version of the classical Fenchel-Moreau duality is applied, which translated to the original framework gives us the representation \eqref{darstellung}. 
The extension procedure consists of three elements. 
First, we conditionally complete the Banach spaces $X$ and $Y$ and prove that they are isometrically isomorphic to the Bochner spaces $L^0(X)$ and $L^0(Y)$, respectively. 
Second, we conditionally extend the duality pairing to a conditional duality pairing on the product $L^0(X)\times L^0(Y)$. 
Third, we construct a conditionally lower semi-continuous extension $f_c$ to $L^0(X)$ with values in $\overline{L^0}$. 
By identifying $\overline{L^0}$ with the conditional extended real numbers and applying a conditional version of the Fenchel-Moreau theorem, we obtain a conditional dual representation of $f_c$ in terms of conditional dual elements in $L^0(Y)$. 
By evaluating constants in the conditional dual representation, we get the vector duality \eqref{darstellung}. 
The conditional completions and conditional extensions and the vector duality result are proved for an arbitrary complete Boolean algebra with the corresponding conditional real numbers instead of an associated measure algebra.  

We discuss the related literature. 
Scalarization techniques, see e.g.~\cite[Chapter 4]{ioan2009duality}, are in general not applicable since $\text{int}(L^0_+)=\emptyset$ and $(L^0)^\ast=\{0\}$. 
If $f\colon X\to L^0$ is convex and continuous with respect to the topology of convergence in measure, \cite[Theorem 1]{zowe75} yields a dual representation $f(x)=\esssup_{x^\ast\in \mathcal{L}(X,L^0)} \{\langle x, x^\ast\rangle - f^\ast(x^\ast)\}$ where $\mathcal{L}(X,L^0)$ is the set of all linear and continuous functions from $X$ to $L^0$ and $f^\ast$ is the corresponding convex conjugate. 
For a convex duality of set-valued functions, we refer to \cite{hamel09,hamel14} and their references. 
Vector duality for modules over the ring $L^0$ is introduced in \cite{kupper03} which is not applicable since $X$ is not a module over $L^0$. 
The module-based duality is further studied in \cite{zapatalebesgue,zapataportfolio} where dual representations of conditional convex risk measures on $L^\infty$-type modules are established. 
The possibility of convex duality in modules is investigated in \cite{kutateladze81}. 
Conditional set theory is developed in \cite{DJKK13} and is closely related to toposes of sheaves over complete Boolean algebras, see \cite{artin1973theorie} and see \cite{lane2012sheaves} for an introduction, and to Boolean-valued models of set theory which emerged from  \cite{cohen63,cohen64} in \cite{scott67,solovay70,vopenka1979mathematics}, see \cite{bell2005set} for an introduction. 
For an overview on Boolean-valued analysis, we refer to \cite{kusraev2012boolean}. 

The article is organized as follows. In Section \ref{sec:metric}, conditional set theory is briefly introduced and conditional completions of metric spaces are constructed. 
In Section \ref{sec:duality}, the conditional extension of dual pairs of normed vector spaces is proved, the concept of lower semi-continuity together with a conditionally lower semi-continuous extension are given and the main vector duality result is shown. 
Finally, the results are applied to vector duality in Bochner spaces in Section \ref{sec:Bochner}. 

\section{Conditional completion of metric spaces}\label{sec:metric}

We start by collecting basic results from conditional set theory as developed in \cite{DJKK13}. 
Throughout, we fix a complete non-degenerate Boolean algebra $\mathcal{A}=(\mathcal{A},\wedge,\vee,{}^c,0,1)$. 
The order on $\mathcal{A}$ is the relation $a\leq b$ whenever $a\wedge b=a$. 
Denote by $\vee a_i=\vee_{i\in I} a_i$ and $\wedge a_i=\wedge_{i\in I} a_i$ the supremum and the infimum of a family $(a_i)=(a_i)_{i\in I}$ in $\mathcal{A}$, respectively. 
A partition in $\mathcal{A}$ is a family $(a_i)$ of elements in $\mathcal{A}$ such that $a_i\wedge a_j=0$ whenever $i\neq j$ and $\vee a_i=1$.
Denote by $p$ the set of all partitions in $\mathcal{A}$. 
\begin{definition}
    A \emph{conditional set} of a non-empty set $X$ and $\mathcal{A}$ is a collection $\mathbf{X}$ of objects $x|a$ where $x\in X$ and $a\in\mathcal{A}$ such that 
    \begin{itemize}
        \item[(C1)] $x|b=y|a$ implies $a=b$; 
        \item[(C2)] $x|b=y|b$ and $a\leq b$ imply $x|a=y|a$; 
        \item[(C3)] for each $(a_i)\in p$ and every $(x_i)$ in $X$ there exists a unique $x\in X$ such that $x|a_i=x_i|a_i$ for all $i$. 
        The unique element $x$ is called the \emph{concatenation} of $(x_i)$ along $(a_i)$ and is denoted by $x=\sum x_i|a_i$. 
    \end{itemize}
\end{definition}
The property (C2) is called \emph{consistency} and the property (C3) is named \emph{stability}. 
A subset $Y\subseteq X$ is \emph{stable} if $Y\neq\emptyset$ and $\sum y_i|a_i\in Y$ for every family $(y_i)$ in $Y$ and each $(a_i)\in p$.
A stable subset $Y$ induces a \emph{conditional subset} $\mathbf{Y}:=\{x|a: x\in Y, a\in\mathcal{A}\}$ of $\mathbf{X}$. 
A singleton $\{x\}$ is stable. The induced conditional subset $\mathbf{x}:=\{x|a\colon a\in \mathcal{A}\}$ is called a \emph{conditional element in} $\mathbf{X}$. 
It is shown in \cite[Theorem 2.9]{DJKK13} that the collection of all conditional subsets of $\mathbf{X}$ with the conditional set operations of \emph{conditional intersection} $\sqcap$, \emph{conditional union} $\sqcup$ and \emph{conditional complement} ${}^\sqsubset$ form a complete Boolean algebra.
A collection $\mathcal{V}$ of conditional subsets of $\mathbf{X}$ is stable if the conditional subset of $\mathbf{X}$ which is induced by 
$\{\sum y_i|a_i\colon y_i\in Y_i \text{ for all }i\}$ is an element of $\mathcal{V}$ for all families $(\mathbf{Y}_i)$ in $\mathcal{V}$ and $(a_i)\in p$. 
The \emph{conditional Cartesian product} of two conditional sets $\mathbf{X}$ and $\mathbf{Y}$ is the conditional set $\mathbf{X}\times \mathbf{Y}:=\{(x,y)\mid a\colon (x,y)\in X\times Y, a\in \mathcal{A}\}$. 
A function $f:X\to Y$ is said to be stable whenever $f(\sum x_i|a_i)=\sum f(x_i)|a_i$ for each $(a_i)\in p$ and every family $(x_i)$ in $X$.
Its graph $G_f:=\{(x,y)\colon f(x)=y\}$ is a stable subset of $X\times Y$.
The induced conditional set is the conditional graph of a \emph{conditional function} $\mathbf{f}:\mathbf{X}\to \mathbf{Y}$.
A conditional function $\mathbf{f}:\mathbf{X}\to \mathbf{Y}$ is \emph{conditionally injective} whenever $f(x)|a=f(y)|a$ implies $x|a=y|a$.
\begin{definition}
    Given a non-empty set $X$, we denote by $X_s$ the set of all families $(x_i,a_i)$ in $X\times \mathcal{A}$ such that $(a_i)\in p$ where two families $(x_i,a_i)$ and $(y_j,b_j)$ are identified whenever $\vee \set{a_i:x_i=z}=\vee \set{b_j:y_j=z}$ for all $z \in X$.
    Denote by $[x_i,a_i]$ the equivalence class of $(x_i,a_i)$ in $X_s$.
    Then $X_s$ induces a conditional set $\mathbf{X_s}$ of objects
    \begin{equation*}
        [x_i,a_i]|a:=\left\{ [y_j,b_j] \in X_s\colon \vee \set{a_i\colon x_i=z}\wedge a=\vee \set{b_j\colon y_j=z}\wedge a \text{ for all }z\in X\right\}. 
    \end{equation*}
    We call $\mathbf{X_s}$ the conditional set of \emph{step functions} with values in $X$.  
    For two non-empty sets $X$ and $Y$, the function $f_s\colon X_s\to Y_s$ defined by $f_s(\sum x_i|a_i):=\sum f(x_i)|a_i$ is a stable function. 
    We call the induced conditional function $\mathbf{f_s\colon \mathbf{X_s}\to Y_s}$ a \emph{conditional step function}. 
\end{definition}
Note that there exists a bijection from $X$ to $\{[x,1]\colon x\in X\}$ as a non-stable subset of $\mathbf{X_s}$ representing the constant step functions in $\mathbf{X_s}$.
By stability, each element $[x_i,a_i]\in X_s$ can be written as $\sum [x_i,1]|a_i$ and therefore the notation $\sum x_i|a_i$ is assigned to the elements of $X_s$.     
The conditional set $\mathbf{N_s}$ of step functions with values in the natural numbers $\mathbb{N}$ is called the \emph{conditional natural numbers}.

    We summarize the construction of the conditional real numbers.
    Let $\mathbf{R_s}$ be the conditional set of step functions with values in the real numbers $\mathbb{R}$. 
    On $\mathbb{R}_s$ define the stable relation $\sum x_i\mid a_i \leqslant \sum y_j\mid b_j$ whenever $x_i\leq y_j$ for all $i,j$ with $a_i\wedge b_j>0$. 
    We understand $\sum x_i\mid a_i < \sum y_j\mid b_j$ if $x_i< y_j$ for all $i,j$ with $a_i\wedge b_j>0$. 
    Moreover, define on $\mathbb{R}_s$ the stable operations $\sum x_i\mid a_i + \sum y_j\mid b_j:=\sum x_i + y_j\mid a_i\wedge b_j$ and $\sum x_i\mid a_i \cdot \sum y_j\mid b_j:=\sum x_i \cdot y_j\mid a_i\wedge b_j$. 
    Then $(\mathbf{R_s},+,\cdot,\leqslant)$ is a conditionally ordered field.\footnote{See \cite[Definition 4.1]{DJKK13}.} 
    Define the stable absolute value $|\sum x_i\mid a_i|:=\sum|x_i|\mid a_i$. 
    A conditional sequence\footnote{A conditional sequence $(\mathbf{x}_{\mathbf{n}})$ in a conditional set $\mathbf{X}$ is a conditional function $\mathbf{N_s}\to\mathbf{X}$.} $(\mathbf{x_n})$ in $\mathbf{R_s}$ is \emph{conditionally Cauchy} if for all $\mathbf{r}$ in $\mathbf{R_s}$ with $\mathbf{r>0}$ there is $\mathbf{n}_0$ in $\mathbf{N_s}$ such that $|\mathbf{x_m-x_n}|< \mathbf{r}$ for all $\mathbf{m,n}\geqslant \mathbf{n}_0$.\footnote{The conditional order on $\mathbf{N_s}$ is inherited from the conditional order on $\mathbf{R_s}$.}  
    By conditionally identifying\footnote{See \cite[Definition 2.15]{DJKK13}.} two conditional Cauchy sequences $(\mathbf{x_n})$ and $(\mathbf{y_n})$ whenever $|\mathbf{x_n-y_n}|$ conditionally converges\footnote{For the concept of conditional convergence, see \cite[Definition 3.19]{DJKK13}.} to $\mathbf{0}$, we obtain the \emph{conditional real numbers} which we denote by $\mathbf{R}$. 
    The conditional elements in $\mathbf{R}$ are denoted $\mathbf{x}=[(\mathbf{x_n})]$. 
    The stable set which induces $\mathbf{R}$ is denoted by $R$ and consists of equivalence classes of stable Cauchy sequences $x=[(x_n)]$. 
    The conditional order on $\mathbf{R}$ is induced by the stable relation on $R$ defined by $[(x_n)]\leqslant [(y_n)]$ whenever for all $r\in \mathbb{R}_s$ with $r>0$ there exists $n\in \mathbb{N}_s$ such that $y_m -x_m>-r$ for all $m\geqslant n$.
    We write $[(\mathbf{x_n})]<[(\mathbf{y_n})]$ if there are $\mathbf{r}$ in $\mathbf{R_s}$ with $\mathbf{r>0}$ and $\mathbf{n}$ in $\mathbf{N_s}$ such that $\mathbf{y_m -x_m>r}$ for all $\mathbf{m}\geqslant \mathbf{n}$. 
    We denote $\mathbf{R}_{+}=\{\mathbf{x}\text{ in }\mathbf{R}\colon \mathbf{x\geqslant 0}\}$ and $\mathbf{R_{++}}=\{\mathbf{x}\text{ in }\mathbf{R\colon x>0}\}$.
    The conditionally ordered set $(\mathbf{R},\leqslant)$ is conditionally Dedekind complete which implies that the ordered set $(R,\leqslant)$ is Dedekind complete in a classical sense.
    The conditional arithmetic operations are defined by $[(\mathbf{x_n})]+[(\mathbf{y_n})]:=[(\mathbf{x_n}+\mathbf{y_n})]$ and $[(\mathbf{x_n})]\cdot[(\mathbf{y_n})]:=[(\mathbf{x_n}\cdot\mathbf{y_n})]$. 
    We have that $(\mathbf{R},+,\cdot,\leqslant)$ is a conditionally ordered field. 
    The conditional absolute value on $\mathbf{R}$ is induced by the stable mapping $|x|=x|a+(-x)|a^c$ where $a=\vee\{b\colon x|b\geqslant 0|b\}$.\footnote{By the well-ordering theorem, consistency and stability it follows that $x|a\geqslant 0|a$.}    
    A conditional open ball with conditional radius $\mathbf{r}$ in $\mathbf{R_{++}}$ around $\mathbf{x}$ is the conditional set $\mathbf{B_r(x)}:=\{\mathbf{y} \text{ in } \mathbf{R}\colon |\mathbf{x-y}|< \mathbf{r}\}$. 
    The collection of all conditional open balls forms a conditional topological base $\mathcal{B}$.\footnote{See \cite[Definition 3.1]{DJKK13}.}
    The conditional topology conditionally generated\footnote{See the paragraph succeeding \cite[Definition 3.1]{DJKK13}.} by $\mathcal{B}$ is conditionally complete. 

    The classical technique of completion extends to general conditional metric spaces.
\begin{definition}
    A \emph{conditional metric} is a conditional function $\mathbf{d}:\mathbf{X}\times \mathbf{X}\to \mathbf{R}_+$ such that
    \begin{enumerate}[label=\textit{(\roman*)}]
        \item $\mathbf{d(x,y)=0}$ if and only if $\mathbf{x=y}$;
        \item $\mathbf{d(x,y)=d(y,x)}$ for every $\mathbf{x}$ and $\mathbf{y}$ in $\mathbf{X}$;
        \item $\mathbf{d(x,y)\leqslant d(x,z)+d(z,y)}$ for every $\mathbf{x}$, $\mathbf{y}$ and $\mathbf{z}$ in $\mathbf{X}$.
    \end{enumerate}
    The pair $(\mathbf{X,d})$ is a conditional metric space.
    The conditional topology on $\mathbf{X}$ is generated by the conditional topological base of conditional sets $\mathbf{B_r(x)}:=\{ \mathbf{y}\text{ in }\mathbf{X}\colon \mathbf{d(x,y)<r}\}$ for $\mathbf{r}\text{ in }\mathbf{R}_{++}$ and $\mathbf{x}$ in $\mathbf{X}$.
    A conditional metric space is called \emph{conditionally complete} if every conditional Cauchy sequence has a conditional limit.
\end{definition}
\begin{theorem}\label{thm:010101}
    Let $(\mathbf{X,d})$ be a conditional metric space. 
    Then there exists a unique, up to conditional isometric isomorphisms, conditionally complete metric space $(\mathbf{X_c},\mathbf{d_c})$ such that $\mathbf{X}$ is conditionally isometric to a conditionally dense subset of $\mathbf{X_c}$.\footnote{See the paragraph succeeding \cite[Definition 3.14]{DJKK13} for the definition of a conditionally dense subset.} 
\end{theorem}
We call $\mathbf{X_c}$ the \emph{conditional completion} of $\mathbf{X}$.
\begin{proof}
    The relation $\mathbf{(x_n)\sim (y_n)}$ whenever $\mathbf{\lim_{n\to\infty} d(x_n,y_n)=0}$ is a conditional equivalence relation on the conditional set $\mathcal{C}$ of all conditional Cauchy sequences in $\mathbf{X}$. 
    Let $\mathbf{X_c}:=\mathcal{C}/_\sim$. 
    For $\mathbf{(x_n)}$ in $\mathcal{C}$, we denote by $\mathbf{[(x_n)]}$ its conditional equivalence class in $\mathbf{X_c}$.
    For $\mathbf{x}$ in $\mathbf{X}$, we denote by $\mathbf{[x]}$ the conditional equivalence class in $\mathbf{X_c}$ of the constant conditional sequence $\mathbf{x_n=x}$ for all $\mathbf{n}$. 
    Let $\mathbf{d_c}\colon \mathbf{X_c}\times \mathbf{X_c}\to \mathbf{R}_+$ be defined by $\mathbf{d_c([(x_n)],[(y_n)])=\lim_{n\to \infty} d(x_n,y_n)}$.
    By the conditional triangle inequality, the conditional completeness of $\mathbf{R}$ and since $\mathbf{d}$ is a conditional metric, $\mathbf{d_c}$ is a well-defined conditional metric on $\mathbf{X_c}$. 
    Define $\mathbf{j\colon X\to X_c}$ by $\mathbf{j(x)=[x]}$. 
    Inspection shows that $\mathbf{j}$ is a conditional isometry.
    In order to see that $\mathbf{j(X)}$ is conditionally dense in $\mathbf{X_c}$, let $[(\mathbf{x_n})]$ be in $\mathbf{X_c}$.
    Since $(\mathbf{x_n})$ is a conditional Cauchy sequence in $\mathbf{X}$, it follows that the conditional sequence $\mathbf{j(x_n)=([x_n])}$ in $\mathbf{j(X)}$ conditionally converges to $\mathbf{[(x_n)]}$ in $\mathbf{X_c}$.
    To prove the conditional completeness of $(\mathbf{X_c,d_c})$, it suffices to show that every conditional Cauchy sequence in $\mathbf{j(X)}$ has a conditional limit in $\mathbf{X_c}$.
    Let $\mathbf{([x_n])}$ be a conditional Cauchy sequence in $\mathbf{j(X)}$.
    Since $\mathbf{j}$ is a conditional isometry, it follows that $\mathbf{([x_n])}$ conditionally converges to $[\mathbf{(x_n)}]$ in $\mathbf{X_c}$.

    As for the uniqueness, let $\mathbf{Y}_1$ and $\mathbf{Y}_2$ be two conditional completions of $\mathbf{X}$ where $\mathbf{j}_1$ and $\mathbf{j}_2$ denote the conditional isometric embeddings of $\mathbf{X}$ into $\mathbf{Y}_1$ and $\mathbf{Y}_2$, respectively. 
    Since $\mathbf{j}_1$ and $\mathbf{j}_2$ are conditionally injective, $\mathbf{h}:=\mathbf{j}_2\circ \mathbf{j}_1^{-1}\colon \mathbf{j}_1(\mathbf{X})\to\mathbf{j}_2(\mathbf{X})$ and $\mathbf{k}:=\mathbf{j}_1\circ \mathbf{j}_2^{-1}\colon \mathbf{j}_2(\mathbf{X})\to\mathbf{j}_1(\mathbf{X})$ are conditional surjective isometries. 
    Now there exist unique conditional isometries $\mathbf{f}\colon \mathbf{Y}_1\to \mathbf{Y}_2$ and $\mathbf{g}\colon\mathbf{Y}_2\to \mathbf{Y}_1$ which conditionally extend $\mathbf{h}$ and $\mathbf{k}$, respectively.\footnote{By similar arguments as in the proof of Proposition \ref{p:uctsext}, it can be shown that there exists a conditionally unique extension.}
    One has $\mathbf{f}\circ \mathbf{j}_1=\mathbf{j}_2$ and $\mathbf{g}\circ \mathbf{j}_2 =\mathbf{j}_1$. 
    Since $\mathbf{j}_1(\mathbf{X})$ is conditionally dense in $\mathbf{Y}_1$ and $\mathbf{j}_2(\mathbf{X})$ is conditionally dense in $\mathbf{Y}_2$, $\mathbf{g\circ f}$ is the conditional identity of $\mathbf{Y}_1$ and $\mathbf{f\circ g}$ is the conditional identity of $\mathbf{Y}_2$. 
    Hence $\mathbf{f}=\mathbf{g}^{-1}$, and thus $\mathbf{f}\colon \mathbf{Y}_1\to\mathbf{Y}_2$ is the unique conditionally surjective isometry such that $\mathbf{f}\circ \mathbf{j}_1=\mathbf{j}_2$. 
\end{proof}
\begin{remark}
    Let $(X,d)$ be a metric space and $\mathbf{d_s\colon X_s\times X_s}\rightarrow \textbf{R}_+$ induced by the stable function 
    \begin{equation*}
            \left(\sum x_i|a_i, \sum y_j|b_j\right)\longmapsto d_s\left(\sum x_i|a_i, \sum y_j|b_j\right):=\sum d(x_i,y_j)|a_i\wedge b_j.
    \end{equation*}
    Inspection shows that $(\mathbf{X_s},\mathbf{d_s})$ is a conditional metric space. 
    By Theorem \ref{thm:010101}, $(\mathbf{X_s},\mathbf{d_s})$ can be conditionally completed into $(\mathbf{X_c,d_c})$.
    We denote by $(\mathbf{X_s,d_s})$ and $(\mathbf{X_c,d_c})$ the conditional metric space of step functions and its conditional completion, respectively. 
\end{remark}
\begin{remark}\label{rem:stepfunc}
    Let $X$ be a metric space.
    Throughout this article, we identify $\mathbf{X_s}$ with $\mathbf{j(X_s)}$ in $\mathbf{X_c}$. 
    For every $x \in X$, we denote by $\mathbf{x}$ the conditional element in $\mathbf{X_c}$ of the constant step function $x|1$.
\end{remark}
\begin{definition}
    A \emph{conditional norm} on a conditional vector space\footnote{See \citep[Definition 5.1]{DJKK13}.}
  $(\mathbf{X},+,\cdot)$ is a conditional function $\norm{\cdot}:\mathbf{X}\to \mathbf{R}_+$ such that
    \begin{enumerate}[label=\textit{(\roman*})]
        \item $\norm{\mathbf{x}}=\mathbf{0}$ if and only if $\mathbf{x}=\mathbf{0}$;
        \item $\norm{\mathbf{r} \mathbf{x}}=|\mathbf{r}|\norm{\mathbf{x}}$ for all $\mathbf{x}$ in $\mathbf{X}$ and every $\mathbf{r}$ in $\mathbf{R}$;
        \item $\norm{\mathbf{x}+\mathbf{y}}\leqslant \norm{\mathbf{x}}+\norm{\mathbf{y}}$ for all $\mathbf{x}$ and $\mathbf{y}$ in $\mathbf{X}$.
    \end{enumerate}
\end{definition}
 \begin{remark}
 By Theorem \ref{thm:010101}, every conditional normed vector space $(\mathbf{X},\|\cdot\|)$ can be conditionally completed into a conditional metric space $\mathbf{X_c}$. 
 Inspection shows that $(\mathbf{X_c,\|\cdot\|_c})$ is a conditional Banach space. Moreover, every normed vector space $(X,\|\cdot\|)$ can be conditionally completed into a conditional Banach space $(\mathbf{X_c,\|\cdot\|_c})$.
\end{remark}
\begin{proposition}\label{p:uctsext}
 Let $(X,d)$ be a metric space and $f\colon X\to\mathbb{R}$ a uniformly continuous function. 
 Then there exists a unique conditionally uniformly continuous function $\mathbf{f_c\colon X_c\to R}$ such that $f_c(x)=f(x)$ for all $x\in X$. 
\end{proposition}
We call $\mathbf{f_c}$ the \emph{conditionally uniformly continuous extension} of $f$.
\begin{proof}
 Define $\mathbf{f_c}\colon \mathbf{X_c}\to \mathbf{R}$ by $\mathbf{f_c(x)}=\lim_{\mathbf{n}\to \infty} \mathbf{f_s(x_n)}$ where $(\mathbf{x_n})$ is a conditional sequence in $\mathbf{X_s}$ conditionally converging to $\mathbf{x}$. 
 We show that $\lim_{\mathbf{n}\to \infty} \mathbf{f_s(x_n)}$ exists. 
 Fix $\mathbf{r}$ in $\mathbf{R}_{++}$ and choose $\mathbf{k}$ in $\mathbf{N_s}$ such that $\mathbf{0<1/k<r}$. 
 Suppose $k=\sum k_i|a_i\in \mathbb{N}_s$. 
 By the uniform continuity of $f$, there exists $t_i\in\mathbb{R}_{++}$ such that $x,y\in X$ with $d(x,y)<t_i$ implies $\vert f(x)-f(y)\vert < 1/k_i$ for each $i$. 
 Put $t=\sum t_i|a_i\in\mathbb{R}_s$. 
 Since $(\mathbf{x_n})$ is a conditional Cauchy sequence there exists $\mathbf{n}_0$ in $\mathbf{N_s}$ such that $\mathbf{d_s(x_n,x_m)<t}$ for all $\mathbf{n,m}\geqslant \mathbf{n}_0$. 
 By construction, $|\mathbf{f_s(x_n)}-\mathbf{f_s(x_m)}| <\mathbf{r}$ for all $\mathbf{n,m}\geqslant \mathbf{n}_0$. 
 The claim follows from the conditional completeness of $\mathbf{R}$. 
 Inspection shows that $\mathbf{f_c}$ is uniquely determined, independent of the choice of conditional representatives and conditionally uniformly continuous. 
\end{proof}
\begin{examples}\label{exp:ucts}
 \begin{itemize}[fullwidth]
  \item[(i)] Let $X$ be a normed vector space and $f\colon X\to \mathbb{R}$ a continuous and linear functional. 
  Its conditionally uniformly continuous extension $\mathbf{f_c\colon X_c\to R}$ is conditionally linear. 
  \item[(ii)] Let $\overline{\mathbb{R}}:=\mathbb{R}\cup\{\pm \infty\}$ be the extended real numbers and let $\overline{\mathbf{R}}$ denote the conditional completion of $\overline{\mathbb{R}}$. 
  The conditionally uniformly continuous extension $\arctan_\mathbf{c}\colon \overline{\mathbf{R}}\to \mathbf{R}$ of  $\arctan\colon\overline{\mathbb{R}}\to \mathbb{R}$ is conditionally strictly increasing. 
\end{itemize}
\end{examples}

\section{Conditional extension of lower semi-continuous functions}\label{sec:duality}

Throughout this section, let $(X,Y,\langle \cdot,\cdot \rangle)$ be a dual pair of normed vector spaces such that $|\langle x,y  \rangle|\leq\|x\|\|y\|$ for all $x\in X$ and $y \in Y$. We assume that $C^X:=\{x\in X\colon \|x\|\leq 1\}$ is $\sigma( X,Y )$-closed and $C^Y:=\{y\in Y\colon \|y\|\leq 1\}$ is $\sigma(Y,X)$-closed. 
\begin{theorem}\label{thm:impo01}
    There exists a unique conditional bilinear form $\langle\cdot,\cdot \rangle_\mathbf{c}$ on $\mathbf{X_c\times Y_c}$ with $\langle x,y \rangle_c=\langle x,y \rangle$ for all $x\in X$ and $y\in Y$ such that $\mathbf{(X_c,Y_c,\langle\cdot,\cdot\rangle{}_c)}$ is a conditional dual pair.\footnote{See \cite[Definition 5.6]{DJKK13}.} 
\end{theorem}
We call $\langle\cdot,\cdot \rangle_\mathbf{c}$ the \emph{conditional extension} of the duality pairing $\langle \cdot,\cdot \rangle$. 
\begin{proof}
    Let $(\mathbf{x_n})$ be a conditional Cauchy sequence in $\mathbf{X_s}$ and $(\mathbf{y_n})$ a conditional Cauchy sequence in $\mathbf{Y_s}$. 
    By the conditional triangle inequality, for all $\mathbf{r>0}$ there is $\mathbf{n}_0\geqslant \mathbf{1}$ such that 
    \begin{equation}\label{eq:impo01}
        \mathbf{\vert\langle x_n,y_n\rangle_s - \langle x_m, y_m\rangle_s \vert\leqslant \Vert x_n - x_m\Vert_s \Vert y_n \Vert_s+ \Vert x_m \Vert_s \Vert y_n - y_m\Vert_s \leqslant r}
    \end{equation}
    for all $\mathbf{n,m}\geqslant\mathbf{n}_0$.
    Put $\langle [(\mathbf{x_n})], [(\mathbf{y_n})] \rangle_{\mathbf{c}}:=\lim_{\mathbf{n}\to \infty}\mathbf{\langle x_n,y_n\rangle_s}$ for $[(\mathbf{x_n})]$ in $\mathbf{X_c}$ and $[(\mathbf{y_n})]$ in $\mathbf{Y_c}$. 
    By \eqref{eq:impo01} and the conditional completeness of the conditional real line, $\langle \cdot,\cdot \rangle_\mathbf{c}$ is independent of conditional representatives and conditionally real-valued, and thus well-defined. 
    Since $\langle \cdot,\cdot\rangle_\mathbf{s}$ is a conditional $\mathbf{R_s}$-bilinear form and since the conditional sum and the conditional product of two conditionally convergent sequences in $\mathbf{R}$ are conditionally convergent, $\langle\cdot,\cdot\rangle_\mathbf{c}$ is a conditional bilinear form on $ \mathbf{X_c}\times\mathbf{Y_c}$.  
    By \eqref{eq:impo01}, the conditional pairing $\langle \cdot,\cdot\rangle_\mathbf{c}$ is the unique conditional bilinear form conditionally extending $\langle\cdot,\cdot\rangle$. 
    
    We show that $\langle\cdot,\cdot\rangle_\mathbf{c}$ conditionally separates points. 
    Let $\mathbf{x}$ be a conditionally non-zero element in $\mathbf{X_c}$, that is, $x|a\neq 0|a$ for every $a>0$. 
    There is $\mathbf{r>0}$ with $r=\sum r_i\mid a_i$ such that $\mathbf{0}$ is conditionally not in the conditional closed ball $\mathbf{C_r(x):=\{z \text{ in } X_c\colon \|z-x\|_c\leqslant r\}}$, that is, $0|a \not \in C_r(x)|a$ for every $a>0$. 
    Since $\mathbf{X_s}$ is conditionally dense in $\mathbf{X_c}$ by Theorem \ref{thm:010101}, we find $\mathbf{x^\prime}$ in $\mathbf{X_s}$ such that $\mathbf{0}$ is conditionally not in $\mathbf{C_{r/2}(x^\prime)}$ and $\mathbf{x}$ is in $\mathbf{C_{r/2}(x^\prime)}$.
    Suppose, without loss of generality, that $x^\prime=\sum x^\prime_i\mid a_i$.\footnote{That is, $x^\prime$ and $r$ have the same partition. Otherwise consider $x^\prime$ and $r$ on the common refinement of the respective partitions.} 
    Note that
    \begin{equation*}
        C_{r/2}(x^\prime)=\left\{z\in X_c\colon z\mid a_i \in C_{r_i/2}(x^\prime_i)| a_i \text{ for all } i\right\}. 
    \end{equation*}
    Since $C_{r_i/2}^X(x_i^\prime):=\{z\in X\colon \| x_i^\prime - z \| \leq r_i/2\}$ is $\sigma(X,Y)$-closed, by strong separation, there exists a non-zero $y_i \in Y$ such that $\inf_{z\in C_{r_i/2}^X(x_i^\prime)}\langle z,y_i\rangle\geq \delta_i>0$. 
    It follows that $\mathbf{\inf_{z\text{ in }C_{r/2}(x^\prime)}\langle z,y\rangle_c\geqslant \delta >0}$ where $y=\sum y_i|a_i$ and $\delta=\sum \delta_i|a_i$.
    Indeed, let $\mathbf{z}$ be in $\mathbf{C_{r/2}(x^\prime)}$ and pick a conditional sequence $\mathbf{(z_n)}$ in $\mathbf{C_{r/2}(x^\prime)}\sqcap \mathbf{X_s}$ such that $\mathbf{\|z_n-z\|_c}\to \mathbf{0}$.
    Since $z_n$ has the form $\sum z_{j,n}|b_{j,n}$, it follows that $\|z_{j,n}-x^\prime_i\|\leq r_i/2$ whenever $b_{j,n}\wedge a_i>0$.
    By stability of $\langle \cdot,\cdot \rangle_s$, we obtain
    \begin{equation*}
        \langle z_n,y\rangle=\left\langle \sum z_{j,n}|b_{j,n},\sum y_{i}|a_i \right\rangle_{s}=\sum \langle z_{j,n},y_i \rangle|b_{j,n}\wedge a_i\geqslant \sum \delta_i|b_{j,n}\wedge a_i=\sum \delta_i|a_i=\delta>0.
    \end{equation*}
    By \eqref{eq:impo01}, it follows that
    \begin{equation*}
        \mathbf{\langle z,y \rangle_c=\langle z-z_n,y \rangle_{c}+\langle z_n,y \rangle_c\geqslant \langle z-z_n,y \rangle_c+\delta \xrightarrow[n\to \infty]{} \delta>0},
    \end{equation*}
    showing that $\mathbf{\inf_{z\text{ in }C_{r/2}(x^\prime)}\langle z,y\rangle_c\geqslant \delta >0}$.
    Hence $\mathbf{Y_s}$, and therefore $\mathbf{Y_c}$, conditionally separates the points of $\mathbf{X_c}$.
    By symmetry, $\mathbf{X_c}$ conditionally separates the points of $\mathbf{Y_c}$.
    Thus $\mathbf{(X_c,Y_c,\langle \cdot,\cdot \rangle_c)}$ is a conditional dual pair.   
\end{proof}
\begin{examples}\label{r:dualpairing}
 \begin{itemize}[fullwidth]
 \item[(i)] If $X=L^p$ and $Y=L^q$ on a finite measure space $(S,\mathcal{S},\mu)$ with $1\leq p,q\leq \infty$ and $1/p+1/q\leq 1$, then $|\langle x,y\rangle|\leq \| x\|_p \|y\|_q$ for all $x\in L^p$ and $y\in L^q$. 
        Moreover, $C^{L^p}$ is $L^1$-closed and therefore $\sigma(L^1,L^\infty)$-closed. Since the identity from $(L^p,\sigma(L^p,L^q))$ to $(L^1,\sigma(L^1,L^\infty))$ is continuous, it follows that $C^{L^p}$ is $\sigma(L^p,L^q)$-closed. That $C^{L^q}$ is $\sigma(L^q,L^p)$-closed works analogously.
  \item[(ii)] In the case that $Y=X^\ast$ is the norm dual of $X$, the assumptions of Theorem \ref{thm:impo01} are satisfied for the duality pairing $\langle x,x^\ast\rangle :=x^\ast(x)$.
    Indeed, since $C^X$ is norm-closed and convex, it is also $\sigma(X,X^\ast)$-closed.
    On the other hand, since $C^{X^\ast}$ is the absolute polar of $C^X$, it follows from the Banach-Alaoglu theorem that $C^{X^\ast}$ is $\sigma(X^\ast,X)$-compact and therefore $\sigma(X^\ast,X)$-closed.
\end{itemize}
\end{examples}
Let $\mathbf{\overline{R}}$ denote the conditional extended real line which is obtained by a conditional completion of the extended real line $\overline{\mathbb{R}}:=\mathbb{R}\cup\{\pm\infty\}$.  
Denote by $\overline{R}$ the stable set which induces $\mathbf{\overline{R}}$.  
We endow $\mathbf{X_c}$ with the conditional $\sigma(\mathbf{X_c,Y_c})$-topology with conditional neighbourhood basis\footnote{For the concept of a conditional neighbourhood basis, see \cite[Definition 3.14]{DJKK13}. A conditionally finite family is defined in \cite[Definitions 2.20 and 2.23]{DJKK13}.} 
\begin{equation*}
 \mathcal{V}(\mathbf{x})=\left\{\mathbf{V^{r,x}_{(y_l)_{1\leqslant l\leqslant n}}}\colon \mathbf{r} \text{ in }\mathbf{R_{++}},\, \mathbf{(y_l)_{1\leqslant l\leqslant n}} \text{ conditionally finite family in } \mathbf{Y_c}\right\}
\end{equation*}
where
\begin{equation*}
 \mathbf{V^{r,x}_{(y_l)_{1\leqslant l\leqslant n}}}:=\{\mathbf{z} \text{ in }\mathbf{X_c}\colon \vert \langle \mathbf{z-x, y_l}\rangle_\mathbf{c} \vert \leqslant \mathbf{r} \text{ for all } \mathbf{1\leqslant l\leqslant n}\}
\end{equation*}
for all $\mathbf{x}$ in $\mathbf{X_c}$. 
Recall that $f_s(\sum x_i|a_i):=\sum f(x_i)|a_i$ for all $\sum x_i|a_i\in X_s$.
\begin{definition}\label{d:lsc}
We say that a function $f\colon X\to \overline{R}$ is 
 \begin{itemize}
  \item \emph{lower semi-continuous} whenever
 \begin{equation}\label{eq:lsc}
  f(x)=\sup_{V\in \mathcal{V}(x)}\inf\{f_s(z)\colon z\in V\cap X_s\}
 \end{equation}
 for all $x\in X$;\footnote{We denote by $V\subseteq X_c$ the stable set which induces $\mathbf{V}\in\mathcal{V}(\mathbf{x})$. Let $\mathcal{V}(x)$ denote the collection of all stable sets $V$ such that $\mathbf{V}\in\mathcal{V}(\mathbf{x})$. Note that $V\cap X_s\neq\emptyset$ for all $V\in\mathcal{V}(x)$ and $x\in X$.}
\item and \emph{proper convex} whenever $-\infty <f(x)$ for all $x\in X$ and $f(x_0)\in R$ for at least one $x_0\in X$ and\footnote{On the right-hand side of the following inequality $\mathbb{R}$ is identified with a non-stable subset of $R$, see Remark \ref{rem:stepfunc}.} $f(\lambda x + (1-\lambda)y)\leqslant \lambda f(x)+(1-\lambda)f(y)$ for all $\lambda\in\mathbb{R}$ with $0\leq \lambda \leq 1$ and every $x,y\in X$.
\end{itemize}
 We say that a conditional function $\mathbf{f\colon X_c\to \overline{R}}$ is 
 \begin{itemize}
  \item \emph{conditionally lower semi-continuous} whenever
 \begin{equation*}
  \mathbf{f}(\mathbf{x})=\sup_{\mathbf{V}\in\mathcal{V}(\mathbf{x})}\mathbf{\inf\{f(z)\colon z\text{ in }V\}} 
 \end{equation*} 
 for all $\mathbf{x}$ in $\mathbf{X_c}$; 
 \item and \emph{conditionally proper convex} whenever $-\infty <\mathbf{f(x)}$ for all $\mathbf{x}$ in $\mathbf{X_c}$ and $\mathbf{f}(\mathbf{x}_0)$ in $\mathbf{R}$ for at least one $\mathbf{x}_0$ in $\mathbf{X_c}$ and $\mathbf{f(\lambda x + (1-\lambda)y)\leqslant \lambda f(x)+(1-\lambda)f(y)}$ for all $\lambda$ in $\mathbf{R}$ with $\mathbf{0\leqslant \lambda \leqslant 1}$ and every $\mathbf{x,y}$ in $\mathbf{X_c}$.
 \end{itemize} 
\end{definition} 
\begin{remark}\label{r:lsc}
The definition of conditional lower semi-continuity is a conditional version of the classical definition of lower semi-continuity for extended real-valued functions on a topological space, see e.g.~\cite[Chapter 6, Section 2]{garling2012topological}, and thus is equivalent to each of the following conditions:
\begin{itemize}
 \item[(i)] $\mathbf{f_c(x)\leqslant \liminf f_c(x_\alpha)}$ for every conditional net\footnote{See \cite[Example 2.21]{DJKK13}.} $(\mathbf{x_\alpha})$ conditionally converging to $\mathbf{x}$;
 \item[(ii)] the conditional lower level set $\{\mathbf{x}\text{ in }\mathbf{X_c}\colon \mathbf{f_c(x)\leqslant r}\}$ is conditionally closed for all $\mathbf{r}$ in $\mathbf{R}$.
\end{itemize}
If the conditional set $\mathbf{X_s}$, seen as a conditional subset of $\mathbf{X_c}$, is endowed with the conditional relative $\sigma(\mathbf{X_c},\mathbf{Y_c})$-topology,\footnote{See \cite[Example 3.9.1]{DJKK13}.} then the lower semi-continuity property \eqref{eq:lsc} expresses the fact that $\mathbf{f_s}$ as a conditional function from $\mathbf{X_s}$ to $\overline{\mathbf{R}}$ is conditionally lower semi-continuous. 
In particular, $f\colon X\to \overline{R}$ is lower semi-continuous if and only if $f(x)\leqslant \liminf f_s(x_\alpha)$ for every stable net $(x_\alpha)$ in $X_s$ converging to $x$ with respect to the classical topology on $X_s$ related to the conditional topology on $\mathbf{X_s}$.\footnote{For the relation between a classical topology on the stable set $X_s$ and a conditional topology on $\mathbf{X_s}$, see \cite[Proposition 3.5]{DJKK13}.}

Let $\mathscr{V}(x)$ denote a classical weak neighbourhood basis of $x\in X$ induced by the dual pair $(X,Y,\langle \cdot,\cdot \rangle)$ and $f\colon X\to\overline{R}$ be a function. 
By inspection, if 
\begin{equation*}
 f(x)=\sup_{V\in \mathscr{V}(x)}\inf\{f(z)\colon z\in V\}
\end{equation*}
for all $x\in X$, then $f$ is lower semi-continuous in the above sense. 
\end{remark}
In the context of Boolean-valued analysis, an extension result for strongly lower semi-continuous functions from a Banach space into the space of continuous functions on an extremally disconnected compact space is stated in \cite[Theorem 5.4]{kutateladze06}.  
The following theorem establishes a conditional extension result for lower semi-continuous functions in the sense of the previous definition which is necessary for the Fenchel-Moreau type vector duality result. 
\begin{theorem}\label{t:lscextension}
 Let $f\colon X\to \overline{R}$ be lower semi-continuous. 
 Then there exists a conditionally lower semi-continuous function $\mathbf{f_c\colon \mathbf{X_c}\to \overline{R}}$ such that $f_c(x)=f(x)$ for all $x\in X$.  
 Moreover, if $f$ is proper convex, then $\mathbf{f_c}$ is conditionally proper convex.  
\end{theorem}
We call $\mathbf{f_c}$ a \emph{conditionally lower semi-continuous extension} of $f$. 
\begin{proof}
We prove that $\mathbf{f_c\colon X_c\to \overline{R}}$ given by 
\begin{equation*}
  \mathbf{f_c(x):=\sup_{V\in \mathcal{V}(x)}\inf\{f_s(z)\colon z \text{ in } V\sqcap X_s\}} 
 \end{equation*}
for all $\mathbf{x}$ in $\mathbf{X_c}$ is a conditionally lower semi-continuous extension of $f$. 
If $\mathbf{f_c}$ were a conditional function, by definition, $f_c(x)=f(x)$ for all $x\in X$ showing that $\mathbf{f_c}$ is a conditional extension.\footnote{If there is a conditional element $\mathbf{x}$ in $\mathbf{V\sqcap X_s}$, then $\mathbf{V\sqcap X_s=V\cap X_s}$. 
For more details on the conditional intersection, we refer to the proof of \cite[Theorem 2.9]{DJKK13}.}
The remaining proof is divided into three steps.
\begin{itemize}[fullwidth]
 \item[Step 1:] 
  We show that for all $\mathbf{x}$ in $\mathbf{X_c}$ and $\mathbf{V}\in \mathcal{V}(\mathbf{x})$ there exists a conditional element $\mathbf{x}^\prime$ in $\mathbf{X_s}$ such that $\mathbf{x}^\prime$ is a conditional element in $\mathbf{V\sqcap X_s}$ which implies that $\mathbf{f_c}$ is a well-defined conditional function since $\mathbf{f_s}$ is a conditional function on the conditional subset $\mathbf{V\sqcap X_s}$ of $\mathbf{X_s}$. 
 Indeed, let $\mathbf{x}$ be in $\mathbf{X_c}$, $(\mathbf{x_m})$ a conditional sequence in $\mathbf{X_s}$ conditionally converging to $\mathbf{x}$ and  $\mathbf{V^{r,x}_{(y_l)_{1\leqslant l\leqslant n}}}\in \mathcal{V}(\mathbf{x})$. 
 By \eqref{eq:impo01}, there is $\mathbf{m_1}$ in $\mathbf{N_s}$ such that $|\langle \mathbf{x_{m_1}-x},\mathbf{y_1} \rangle_\mathbf{c}|\leqslant \mathbf{r}$. 
 Suppose $\mathbf{n}$ is induced by the step function $\sum n_i|a_i$. 
 Fix $i$ and choose $\mathbf{m}_i\geqslant \mathbf{m_1}$ in $\mathbf{N_s}$ such that $|\langle \mathbf{x}_{\mathbf{m}_i}-\mathbf{x},\mathbf{y_l} \rangle_\mathbf{c}|\leqslant \mathbf{r}$ for all $\mathbf{l}$ induced by step functions of the form $l|a_i + 1|a_i^c$ with $1\leq l \leq n_i$. 
 Set $m=\sum m_i|a_i$. 
 By stability, $|\langle \mathbf{x_{m} - x,y_l} \rangle_\mathbf{c}|\leqslant \mathbf{r}$ for all $\mathbf{1\leqslant l\leqslant n}$. 
 Hence $\mathbf{x^\prime:=x_m}$ in $\mathbf{X_s}$ is in $\mathbf{V^{r,x}_{(y_l)_{1\leqslant l\leqslant n}}}$. 
 \item[Step 2:] We show that $\mathbf{f_c}$ is conditionally lower semi-continuous. 
Let $\mathbf{x}$ be in $\mathbf{X_c}$ and $\mathbf{h}$ in $\mathbf{R}_{++}$. 
 Since $\arctan_\mathbf{c}(\mathbf{f_c}(\mathbf{x})) -\mathbf{h}<\arctan_\mathbf{c}(\mathbf{f_c}(\mathbf{x}))$, there exists $\mathbf{V=V^{r,x}_{(y_l)_{1\leqslant l\leqslant n}}}\in \mathcal{V}(\mathbf{x})$ such that
 \begin{equation*}
 \arctan_\mathbf{c}(\mathbf{f_c}(\mathbf{x})) -\mathbf{h}\leqslant \inf\{\mathbf{\arctan_\mathbf{c}(f_s(z))}\colon \mathbf{z} \text{ in } \mathbf{V}\sqcap \mathbf{X_s}\}.
 \end{equation*}
 Let $\mathbf{w}$ be in $\mathbf{V^{r/2,x}_{(y_l)_{1\leqslant l\leqslant n}}}$. 
 By the conditional triangle inequality, $\mathbf{V}_0\sqsubseteq \mathbf{V}$ for $\mathbf{V}_0=\mathbf{V^{r/2,w}_{(y_l)_{1\leqslant l\leqslant n}}}\in \mathcal{V}(\mathbf{w})$.  
 From $\mathbf{V}_0\sqcap \mathbf{X_s}\sqsubseteq \mathbf{V}\sqcap \mathbf{X_s}$ it follows that 
 \begin{align*}
 \arctan_\mathbf{c}(\mathbf{f_c}(\mathbf{x})) -\mathbf{h} \leqslant \inf\{\mathbf{\arctan_\mathbf{c}(f_s(z))}\colon \mathbf{z} \text{ in } \mathbf{V}_0\sqcap \mathbf{X_s}\} 
 \end{align*}
 and thus
 \begin{align*}
 \arctan_\mathbf{c}(\mathbf{f_c}(\mathbf{x})) -\mathbf{h}\leqslant\mathbf{\arctan_c(f_c(w))}. 
 \end{align*}
 By the previous argument, for each $\mathbf{h}$ in $\mathbf{R}_{++}$ there exists $\mathbf{V_h}\in \mathcal{V}(\mathbf{x})$ such that $\arctan_\mathbf{c}(\mathbf{f_c}(\mathbf{x})) -\mathbf{h} \leqslant \mathbf{\arctan_c(f_c(w))}$ for all $\mathbf{w}$ in $\mathbf{V_h}$. 
 This implies that 
 \begin{equation*}
 \mathbf{\arctan_c(f_c(x)) - h}\leqslant  \mathbf{\sup_{V\in \mathcal{V}(x)}\inf\{\arctan_c(f_c(z))\colon z\text{ in }V\}}.
 \end{equation*}
 By letting $\mathbf{h\downarrow 0}$ and since $\mathbf{\sup_{V\in \mathcal{V}(x)}\inf\{\arctan_c(f_c(z))\colon z\text{ in } V\}\leqslant \arctan_c(f_c(x))}$ is trivially satisfied, it follows from the conditional strict monotonicity of $\arctan_\mathbf{c}$ that 
  \begin{equation*}
 \mathbf{f_c(x) = \sup_{V\in \mathcal{V}(x)}\inf\{f_c(z)\colon z\text{ in } V\}}.
 \end{equation*}
 \item[Step 3:] First, we show that $\mathbf{f_c}$ is conditionally convex. 
 For $\mathbf{x},\mathbf{x}^\prime$ in $\mathbf{X_s}$ and $\lambda$ in $\mathbf{R_s}$ with $\mathbf{0\leqslant \lambda \leqslant 1}$, it follows from the hypothesis that 
 \begin{equation*}
  \mathbf{f_s}(\lambda \mathbf{x} + (1-\lambda) \mathbf{x}^\prime)\leqslant \lambda \mathbf{f_s}(\mathbf{x}) + (1-\lambda) \mathbf{f_s}(\mathbf{x}^\prime). 
 \end{equation*}
Let $\mathbf{x},\mathbf{x}^\prime$ be in $\mathbf{X_c}$ and $\lambda$ in $\mathbf{R_s}$ with $\mathbf{0\leqslant \lambda \leqslant 1}$. 
Let $\mathbf{V}=\mathbf{V^{r,\lambda \mathbf{x}+(1-\lambda)\mathbf{x}^\prime}_{(y_l)_{1\leqslant l \leqslant n}}}\in \mathcal{V}(\lambda \mathbf{x}+(1-\lambda)\mathbf{x}^\prime)$, $\mathbf{W}=\mathbf{V}^{\mathbf{r},\mathbf{x}}_{\mathbf{(y_l)_{1\leqslant l \leqslant n}}}\in \mathcal{V}(\mathbf{x})$ and $\mathbf{W}^\prime=\mathbf{V}^{\mathbf{r},\mathbf{x}^\prime}_{\mathbf{(y_l)_{1\leqslant l \leqslant n}}}\in \mathcal{V}(\mathbf{x}^\prime)$. 
It can be checked that if $\mathbf{z}$ is in $\mathbf{W}\sqcap \mathbf{X_s}$ and $\mathbf{z}^\prime$ is in $\mathbf{W}^\prime\sqcap \mathbf{X_s}$, then $\lambda \mathbf{z} + (\mathbf{1}-\lambda)\mathbf{z}^\prime$ is in $\mathbf{V}\sqcap \mathbf{X_s}$. 
Therefore, one has 
\begin{align*}
 \mathbf{\inf\{f_s(z)\colon \mathbf{z}\text{ in }\mathbf{V\sqcap X_s}\}}&\leqslant\inf\{\mathbf{f_s}(\lambda \mathbf{z} + (\mathbf{1}-\lambda)\mathbf{z}^\prime)\colon \mathbf{z}\text{ in }\mathbf{W}\sqcap \mathbf{X_s}, \, \mathbf{z}^\prime\text{ in }\mathbf{W}^\prime\sqcap \mathbf{X_s}\}\\
 &\leqslant \inf\{\lambda \mathbf{f_s}(\mathbf{z})+(\mathbf{1}-\lambda)\mathbf{f_s}(\mathbf{z}^\prime)\colon \mathbf{z}\text{ in }\mathbf{W}\sqcap \mathbf{X_s},\, \mathbf{z}^\prime\text{ in }\mathbf{W}^\prime\sqcap \mathbf{X_s}\}\\
 &=\lambda \inf\{\mathbf{f_s}(\mathbf{z})\colon \mathbf{z}\text{ in }\mathbf{W}\sqcap \mathbf{X_s}\} + (\mathbf{1}-\lambda) \inf\{\mathbf{f_s}(\mathbf{z}^\prime)\colon \mathbf{z}^\prime\text{ in }\mathbf{W}^\prime\sqcap \mathbf{X_s}\}. 
\end{align*}
We found for every $\mathbf{V}\in\mathcal{V}(\lambda \mathbf{x}+(\mathbf{1}-\lambda)\mathbf{x}^\prime)$ conditional neighbourhoods $\mathbf{W}\in\mathcal{V}(\mathbf{x})$ and $\mathbf{W}^\prime\in \mathcal{V}(\mathbf{x}^\prime)$ such that 
\begin{equation*}
 \mathbf{\inf\{f_s(z)\colon \mathbf{z}\text{ in }\mathbf{V\sqcap X_s}\}}\leqslant \lambda \inf\{\mathbf{f_s}(\mathbf{z})\colon \mathbf{z}\text{ in }\mathbf{W}\sqcap \mathbf{X_s}\} + (\mathbf{1}-\lambda) \inf\{\mathbf{f_s}(\mathbf{z})\colon \mathbf{z}\text{ in }\mathbf{W}^\prime\sqcap \mathbf{X_s}\}. 
\end{equation*}
By taking the conditional supremum on both sides of the previous conditional inequality, one obtains 
\begin{equation*}
 \mathbf{f_c}(\lambda \mathbf{x} + (\mathbf{1}-\lambda)\mathbf{x}^\prime)\leqslant \lambda \mathbf{f_c}(\mathbf{x}) + (\mathbf{1}-\lambda)\mathbf{f_c}(\mathbf{x}^\prime).  
\end{equation*}
The last conditional inequality holds for $\lambda$ in $\mathbf{R}$ with $\mathbf{0\leqslant \lambda \leqslant 1}$ by approximating $\lambda$ with step functions in $\mathbf{R_s}$ and the conditional lower semi-continuity of $\mathbf{f_c}$. 

Second, we show that $\mathbf{f_c}$ is conditionally proper. 
Since $f$ is proper and $\mathbf{f_c}$ is a conditional extension there is $\mathbf{x}_0$ in $\mathbf{X_c}$ with $\mathbf{f_c}(\mathbf{x}_0)$ in $\mathbf{R}$. 
By way of contradiction, suppose there exist $\mathbf{x}$ in $\mathbf{X_c}$ and $a>0$ such that $f_c(x)|a=-\infty|a$. 
By conditional convexity, one has $f_c(\lambda x_0 + (1-\lambda)x)|a=-\infty|a$ for all $\mathbf{0\leqslant \lambda <1}$. 
Since $\lambda \mathbf{x}_0 + (\mathbf{1}-\lambda)\mathbf{x}$ conditionally converges to $\mathbf{x}_0$ as $\lambda$ conditionally converges to $\mathbf{1}$, $f_c(x_0)|a=-\infty|a$ due to the conditional lower semi-continuity of $\mathbf{f_c}$ which is the desired contradiction. 
\end{itemize}
\end{proof}
As an application of the previous results, we obtain the following Fenchel-Moreau type duality for vector-valued functions.
\begin{theorem}\label{t:FM}
Let $f\colon X\to \overline{R}$ be proper convex. 
Then $f$ is lower semi-continuous if and only if 
\begin{equation}\label{rep1}
 f(x)=\sup_{y\in Y_c}\{\langle x,y\rangle_c - f^\ast(y)\} 
\end{equation}
for all $x\in X$ and for the convex conjugate $f^\ast(y):=\sup_{x\in X} \{\langle x,y\rangle_c - f(x)\}$ for all $y\in Y_c$.
\end{theorem}
\begin{proof}
Suppose $f$ is lower semi-continuous. 
By Theorem \ref{t:lscextension}, there exists a conditionally proper convex lower semi-continuous extension $\mathbf{f_c\colon X_c\to\overline{R}}$. 
By a conditional version of the fundamental theorem of duality,\footnote{In \cite[Section 5]{DJKK13}, a conditional version of basic results in functional analysis is established.} one has 
\begin{equation*}
(\mathbf{X_c},\sigma(\mathbf{X_c,Y_c}))^\ast=\mathbf{Y_c}. 
\end{equation*}
By applying a conditional version of strong separation \cite[Theorem 5.5]{DJKK13}, and following the arguments in the proof of the classical Fenchel-Moreau theorem, one obtains 
\begin{equation*}
 \mathbf{f_c(x)=\sup_{y\text{ in }Y_c}\{\langle x,y\rangle_c - f^\ast_c(y)\}}
\end{equation*}
for all $\mathbf{x}$ in $\mathbf{X_c}$ and for the conditional convex conjugate $\mathbf{f^\ast_c(y):=\sup_{x\text{ in }X_c} \{\langle x,y\rangle_c - f_c(x)\}}$ for all $\mathbf{y}$ in $\mathbf{Y_c}$. 
Fix $x\in X$. 
Since 
\begin{equation}\label{eq19}
 f^\ast(y)=\sup_{x\in X} \{\langle x,y\rangle_c - f(x)\}\leqslant \sup_{x\in X_c} \{\langle x,y\rangle_c - f_c(x)\}=f^\ast_c(y)
\end{equation}
and $\langle x,y\rangle_c - f^\ast(y)\leqslant f(x)$ for all $y\in Y_c$, one has
\begin{align*}
 f(x)=f_c(x)&=\sup_{y\in Y_c}\{\langle x,y\rangle_c - f^\ast_c(y)\}\\
 &\leqslant \sup_{y\in Y_c}\{\langle x,y\rangle_c - f^\ast(y)\} \leqslant f(x). 
\end{align*}
Conversely, suppose \eqref{rep1} holds. 
Fix $x\in X$. 
We show that 
\begin{equation*}
 f(x)=\sup_{V\in\mathcal{V}(x)}\inf\{f_s(z)\colon z\in V\cap X_s\}. 
\end{equation*}
One has $f(x)\geqslant\sup_{V\in\mathcal{V}(x)}\inf\{f_s(z)\colon z\in V\cap X_s\}$ since $x\in V\cap X_s$ for all $V\in \mathcal{V}(x)$. 
By contradiction, suppose there exist $a>0$ and $\delta\in R_{++}$ such that $f(x)|a - \delta|a>\sup_{V\in\mathcal{V}(x)}\inf\{f_s(z)\colon z\in V\cap X_s\}|a$. 
By consistency and stability, we can assume that $a=1$. 
By stability, there exists a stable net $(x^V=\sum x^V_i|a^V_i)_{V\in\mathcal{V}(x)}$ converging to $x$ such that 
\begin{align*}
f(x)- \delta \geqslant f_s(x^V)&=\sum f(x^V_i)|a^V_i \\
		       &=\sum \sup_{y\in Y_c} \{\langle x^V_i,y\rangle_c - f^\ast(y)\}|a_i\\
		       &\geqslant \sup_{y\in Y_c}\{\langle x^V,y\rangle_c - f^\ast(y)\}, 
\end{align*}
where the last inequality follows from
\begin{align*}
 \langle x^V,y\rangle_c - f^\ast(y)&=\sum \langle x^V_i,y\rangle_c|a_i - f^\ast(y)\\
 &=\sum (\langle x^V_i,y\rangle_c-f^\ast(y))|a_i \\
 &\leqslant \sum \sup_{y\in Y_c}\{\langle x^V_i,y\rangle_c - f^\ast(y)\}|a_i 
\end{align*}
for all $y\in Y_c$. 
One concludes
\begin{equation*}
f(x)- \delta \geqslant \langle x^V,y\rangle_c - f^\ast(y) 
\end{equation*}
for all $V\in\mathcal{V}(x)$ and $y\in Y_c$. 
By passing to the limit and taking the supremum, one has
\begin{equation*}
 f(x)- \delta \geqslant \sup_{y \in Y_c}\{\langle x,y\rangle_c - f^\ast(y)\}=f(x) 
\end{equation*} 
which is contradictory. 
\end{proof}
\begin{remark}\label{r:maximality}
 The conditional function $\mathbf{f_c}$ is a conditionally proper convex lower semi-continuous extension of $f$ in the \emph{maximal sense}, that is, for every conditionally proper convex lower semi-continuous extension $\mathbf{g}$ of $f$ one has $\mathbf{g}\leqslant \mathbf{f_c}$. 
 In fact, one has $f^\ast_c\leqslant g^\ast$ similarly to the argument in \eqref{eq19}, and thus
 \begin{align*}
   \mathbf{f_{c}(x)}&=\mathbf{\sup_{y\text{ in }Y_c}\{\langle x,y\rangle_c - f^\ast_c(y)\}} \\
 		    &\geqslant \mathbf{\sup_{y\text{ in }Y_c}\{\langle x,y\rangle_c - g^\ast(y)\}} =\mathbf{g(x)}
 \end{align*}
where the last equality follows by a conditional version of the Fenchel-Moreau theorem as argued in the previous proof. 
\end{remark}
\begin{remark}\label{r:norm}
Let $X$ be a normed vector space and $f\colon X\to \overline{R}$ a proper convex function. 
With similar arguments as in the proofs of Theorem \ref{t:lscextension} and Theorem \ref{t:FM}, one can show that $f$ is lower semi-continuous in the sense that 
\begin{equation}\label{eq:normlsc}
 f(x)=\sup_{n\geq 1}\inf \{f(z)\colon z\in C^X_{1/n}(x)\}
\end{equation}
for all $x\in X$, where $C^X_r(x)$ is the closed ball of radius $r$ around $x$, if and only if $f$ admits the dual representation 
\begin{equation*}
 f(x)=\sup_{x^\ast\in (X^\ast)_c}\{\langle x,x^\ast\rangle_c - f^\ast(x^\ast)\}
\end{equation*}
where $f^\ast(x^\ast)=\sup_{x\in X} \{\langle x,x^\ast\rangle_c - f(x)\}$. 
\end{remark}
\begin{example}
 Let $\mathcal{A}$ be the power set algebra of a finite set of cardinality $d$ and $f\colon X\to \overline{\mathbb{R}}^d$ proper convex and lower semi-continuous. 
 By Theorem \ref{t:FM}, one obtains the following dual representation
 \begin{equation*}
  f(x)=\sup_{y\in Y^d}\{(\langle x,y_1\rangle, \ldots, \langle x, y_d\rangle) - f^\ast(y)\}
 \end{equation*}
where $f^\ast(y)=\sup_{x\in X}\{(\langle x,y_1\rangle, \ldots, \langle x, y_d\rangle) - f(x)\}$ for all $y\in Y^d$. 
\end{example}

\section{Application to vector duality in Bochner spaces}\label{sec:Bochner}

Throughout, let $(X,Y,\langle \cdot,\cdot \rangle)$ be a dual pair of Banach spaces such that $|\langle x,y  \rangle|\leq\|x\|\|y\|$ for all $x\in X$ and $y \in Y$. We assume that $C^X:=\{x\in X\colon \|x\|\leq 1\}$ is $\sigma( X,Y )$-closed and $C^Y:=\{y\in Y\colon \|y\|\leq 1\}$ is $\sigma(Y,X)$-closed. 

Let $(\Omega,\mathcal{F},\mu)$ be a complete $\sigma$-finite measure space. 
Identify in $\mathcal{F}$ two elements the symmetric difference of which is a $\mu$-null set, and thus obtain the associated measure algebra $\mathcal{A}$. 
The associated measure algebra is a complete Boolean algebra which satisfies the countable chain condition.\footnote{For more details about the associated measure algebra, see \cite[Chapter 22, Section 2]{monk1989handbook}.}
We denote the equivalence classes in $\mathcal{A}$ by $a=[A]$ where $A\in\mathcal{F}$. 
We write $1_A$ for the characteristic function of $A\in\mathcal{F}$. 
Let $\overline{L^0}$ be the collection of all measurable functions $x\colon \Omega\to \overline{\mathbb{R}}$ where two of them are identified if they agree almost everywhere. 
As usual, denote by $L^0$ the subset of $\overline{L^0}$ consisting of real-valued measurable functions. 
Henceforth, equalities and inequalities between measurable functions are understood in the almost everywhere sense. 

The set $\overline{L^0}$ induces a conditional set $\mathbf{\overline{L^0}}$ of objects 
\begin{equation}\label{eq:Bochner}
x\mid a=\{y\in \overline{L^0}\colon x 1_A = y 1_A \text{ for some }A\in a\} 
\end{equation}
where $x\in \overline{L^0}$ and $a\in\mathcal{A}$. 
We denote by $\mathbf{L^0}$ the conditional subset of $\mathbf{\overline{L^0}}$ induced by the stable subset $L^0$. 
By the countable chain condition, the elements in $p$ have at most countably many non-zero entries.
Let $(a_n)$ be a sequence in $p$ and $(x_n)$ a sequence in $\overline{L^0}$. 
The concatenation of $(x_n)$ along $(a_n)$ is the equivalence class in $\overline{L^0}$ of the measurable function $\sum x_n 1_{A_n}$ where $a_n=[A_n]$ for all $n$. 
In \cite[Theorem 4.4]{DJKK13}, it is shown that the conditional set $\mathbf{L^0}$ is conditionally isometrically isomorphic to the conditional real numbers $\mathbf{R}$ whenever the latter is constructed with respect to the associated measure algebra.
The stable order on $\overline{R}$ corresponds to the order of almost everywhere dominance on $\overline{L^0}$. 
The infimum and the supremum in $\overline{R}$ correspond to the essential infimum and the essential supremum in $\overline{L^0}$, respectively.  

Recall that a function $x\colon \Omega\to X$ is strongly measurable if there exists a sequence of simple functions $(x_n)$ such that $\| x_n - x \|\to 0$ almost everywhere. 
We denote by $L^0(X)$ the Bochner space of equivalence classes of all strongly measurable functions with respect to the equivalence relation of almost everywhere equality.\footnote{For more details about Bochner spaces, see \cite[Chapter II, Section 2]{joseph1977vector}.}
The Bochner space $L^0(X)$ induces a conditional set $\mathbf{L^0(X)}$ similarly to \eqref{eq:Bochner}. 
The following theorem generalizes \cite[Theorem 4.4]{DJKK13} and \cite[Chapter 7, Theorem 7.1]{bell2005set} in the Boolean-valued context, respectively.
\begin{theorem}\label{p:Bochner}
The conditional set $\mathbf{L^0(X)}$ is a conditional Banach space conditionally isometrically isomorphic to $(\mathbf{X_c},\|\cdot \|_\mathbf{c})$.
\end{theorem}
\begin{proof}
By \cite[Theorem 4.4]{DJKK13}, the conditional real line $\mathbf{R}$ is conditionally isometrically isomorphic to $\mathbf{L^0}$. 
On $\mathbf{L^0(X)}$ the conditional addition and conditional scalar multiplication are induced by the stable functions 
$(x+y)(\omega)=x(\omega)+y(\omega)$ and $(\alpha x)(\omega)=\alpha(\omega)x(\omega)$ almost everywhere, respectively. 
Let $\mathbf{\|\cdot \|_0\colon L^0(X)\to L^0}$ be the conditional function induced by the stable function $\|\cdot \|_0\colon L^0(X)\to L^0$ defined by $\|x\|_0(\omega)=\|x(\omega)\|$ for almost all $\omega\in\Omega$. 
Direct verification shows that $(\mathbf{L^0(X),\|\cdot\|_0})$ is a conditional normed vector space. 
We prove that $\mathbf{L^0(X)}$ is conditionally complete. 
To this end, let $(\mathbf{x_n})$ be a conditional Cauchy sequence in $\mathbf{L^0(X)}$. 
By induction, choose for every $k\in\mathbb{N}$ an $n_k\in\mathbb{N}_s$ such that $n_k\geqslant n_{k-1}$ and $\|x_n-x_m\|_0\leqslant 1/k$ for all $n,m\geqslant n_k$. 
Then $(x_{n_k(\omega)}(\omega))$ is a Cauchy sequence in $X$ almost everywhere. 
By the completeness of $(\Omega,\mathcal{F},\mu)$ and the completeness of $X$, the almost everywhere limit $x(\omega):=\lim_{k\to \infty} x_{n_k(\omega)}(\omega)$ exists in $L^0(X)$.  
By the conditional triangle inequality, $(\mathbf{x_n})$ conditionally converges to $\mathbf{x}$. 
Let $\mathbf{j\colon L^0(X)\to X_c}$ be the conditional function induced by the stable map $j(x):=[(x_n)]$.\footnote{By Pettis measurability theorem in the form of \cite[Chapter II, Section 1, Corollary 3]{joseph1977vector}, every strongly measurable function is the almost sure uniform limit of a sequence $(x_n)$ of countably valued strongly measurable functions. 
Put $x_n=\sum x_{n_m}|a_m$ for every $n\in\mathbb{N}_s$ with $n=\sum n_m|a_m$. Then 
$(\mathbf{x_n})$ is a conditional Cauchy sequence in $\mathbf{X_s}$.} 
By inspection, $\mathbf{j}$ is a conditional isometric isomorphy.
\end{proof}
Note that the duality pairing $\langle \cdot, \cdot \rangle$ on $X\times Y$ extends to $L^0(X)\times L^0(Y)$ by defining 
\begin{equation*}
 \langle x, y\rangle = \lim_{n\to \infty} \langle x_n, y_n \rangle
\end{equation*}
where $x\in L^0(X)$ is the limit of a sequence of simples functions $(x_n)$ and $y\in L^0(Y)$ is the limit of a sequence of simples functions $(y_n)$. 
It can be checked that 
\begin{equation*}
\langle x , y \rangle =\langle x, y\rangle _c
\end{equation*}
for all $x\in L^0(X)$ and $y\in L^0(Y)$ where we identified $X_c$ with $L^0(X)$, $Y_c$ with $L^0(Y)$ and $R$ with $L^0$ by Theorem \ref{p:Bochner} and \cite[Theorem 4.4]{DJKK13}. 
 
The stable set $X_s$ can be identified with 
\begin{equation*}
 \left\{\sum x_n 1_{A_n} \colon (x_n) \text{ in } X, (A_n) \text{ is partition of $\Omega$ in } \mathcal{F}\right\} \subseteq L^0(X),  
\end{equation*}
and $R_{++}$ can be identified with  $L^0_{++}:=\{r\in L^0\colon r>0\}$.  
Recall that a function $f\colon X\to\overline{L^0}$ is lower semi-continuous if $\langle x_\alpha, y\rangle_c$ converges\footnote{That is, for all $r\in L^0_{++}$ there exists $\alpha_0$ such that $|\langle x_\alpha - x,y\rangle_c|\leqslant r$ for all $\alpha\geqslant \alpha_0$.} to $\langle x,y\rangle_c$ implies $f(x)\leqslant \essliminf f_s(x_\alpha)$ for all stable nets\footnote{That is, $\sum x_{\alpha_n}1_{A_n}=x_{\sum \alpha_n 1_{A_n}}$ for all partitions $(A_n)$ of $\Omega$ in $\mathcal{F}$ and countable subfamilies $(x_{\alpha_n})$ of $(x_\alpha)$.} in $X_s$ and $y\in Y_c$. 
The function $f$ is proper convex if $-\infty <f(x)$ for all $x\in X$ and $f(x_0)\in L^0$ for at least one $x_0\in X$ and $f(\lambda x + (1-\lambda)y)\leqslant \lambda f(x)+(1-\lambda)f(y)$ for all $\lambda\in\mathbb{R}$ with $0\leq \lambda \leq 1$ and every $x,y\in X$. 
A reformulation of Theorem \ref{t:FM} in the present context is the following main result in vector duality. 
\begin{theorem}\label{t:vectorduality}
Let $f\colon X\to \overline{L^0}$ be proper convex. Then $f$ is lower semi-continuous if and only if  
\begin{equation}\label{eq:fm}
 f(x)=\esssup_{y\in L^0(Y)}\{\langle x,y\rangle - f^\ast(y)\}
\end{equation}
$\text{ for all }x\in X$ where $f^\ast(y)=\esssup_{x\in X}\{\langle x,y\rangle - f(x)\}$ for all $y\in L^0(Y)$. 
\end{theorem}
\begin{proof}
 The result is a consequence of Theorem \ref{t:FM} and Theorem \ref{p:Bochner}.
\end{proof}
\begin{remark}
 Let $(\Omega,\mathcal{F},\mu)$ be a complete atomless finite measure space, $L^\infty=L^\infty(\Omega,\mathcal{F},\mu)$ and $L^1=L^1(\Omega,\mathcal{F},\mu)$. 
 Consider the identity map
 \begin{equation*}
 I\colon L^\infty\to L^0.
 \end{equation*}
Although the lower level set $\{x\in L^\infty\colon I(x)=x\leqslant r\}$ is $\sigma(L^\infty,L^1)$-closed for all $r\in L^0$ due to the Krein-\v{S}mulian theorem, the identity is not lower semi-continuous. In fact, for all $x\in L^\infty$ and $V\in \mathcal{V}(x)$ one has 
\begin{equation*}
 \essinf\{I_s(z)\colon z\in V\cap L^\infty_s\}=-\infty. 
\end{equation*}
In particular, the identity does not admit a Fenchel-Moreau type representation of the form \eqref{eq:fm}. 
\end{remark}

\end{document}

%% file: head2013.tex
\usepackage[normalem]{ulem}

\newcommand{\norm}[1]{\left\Vert#1\right\Vert}

\newcommand{\set}[1]{\ensuremath{ \{ #1 \} }}

\renewcommand{\mid}{\,|\,}

\DeclareMathOperator*{\esssup}{ess\,sup}
\DeclareMathOperator*{\essinf}{ess\,inf}

\DeclareMathOperator*{\essliminf}{ess\,liminf}